\documentclass[12pt,reqno]{amsart}
\usepackage{amssymb,amsmath,amsthm,fullpage,enumerate}
\usepackage{xcolor}
\usepackage{biblatex} 
\usepackage[OT2,OT1]{fontenc}

\DeclareFontFamily{OT2}{cmr}{\hyphenchar\font45 }
\DeclareFontShape{OT2}{cmr}{m}{n}{<->wncyr10}{}
\DeclareFontShape{OT2}{cmr}{m}{it}{<->wncyi10}{}
\DeclareFontShape{OT2}{cmr}{m}{sc}{<->wncysc10}{}
\DeclareFontShape{OT2}{cmr}{b}{n}{<->wncyb10}{}
\DeclareFontShape{OT2}{cmr}{bx}{n}{<->ssub*wncyr/b/n}{}

\DeclareFontFamily{OT2}{cmss}{\hyphenchar\font45 }
\DeclareFontShape{OT2}{cmss}{m}{n}{<->wncyss10}{}

\DeclareRobustCommand\cyr{\fontencoding{OT2}\selectfont}
\DeclareTextFontCommand{\textcyr}{\cyr}

\DeclareUnicodeCharacter{0306}{}

\newtheorem{theorem}{Theorem}[section]
\newtheorem{prop}[theorem]{Proposition}

\newtheorem{corr}[theorem]{Corollary}

\theoremstyle{definition}

\newtheorem{deff}[theorem]{Definition}

\newtheorem{examp}[theorem]{Example}

\DeclareMathOperator{\supp}{supp}

\newcommand{\p}{\partial}

\newcommand{\z}{\bar z}

\DeclareMathOperator{\im}{Im}
\DeclareMathOperator{\re}{Re}

\newcommand{\Hp}{H^p_{f}(D)}

\newcommand{\Har}{H^p(D)}

\newcommand{\dc}{\mathcal{D}'(\p D)}

\newcommand{\Hpq}{H^{p}_{f,q}(D)}
\newcommand{\Hpqn}{H^{n,p}_{f,q}(D)}

\title{An atomic representation for Hardy classes of solutions to nonhomogeneous Cauchy-Riemann equations}

\author{William L. Blair}
\address{Department of Mathematical Sciences\\
  University of Arkansas\\
  Fayetteville, Arkansas}
\email{wlblair@uark.edu}

\keywords{Hardy spaces, atomic decomposition, generalized analytic functions, nonhomogeneous Cauchy-Riemann equations, boundary value in the sense of distributions, Hilbert transform, Schwarz boundary value problem}

\subjclass[2010]{30H10, 30G30, 35C15, 30E25, 35F30, 35G30, 30G20, 30E20, 46F20, 45E05, 35F15, 35G15, 30J99, 46F99}

\begin{document}

\begin{abstract}
    We develop a representation of the second kind for certain Hardy classes of solutions to nonhomogeneous Cauchy-Riemann equations and use it to show that boundary values in the sense of distributions of these functions can be represented as the sum of an atomic decomposition and an error term. We use the representation to show continuity of the Hilbert transform on this class of distributions and use it to show that solutions to a Schwarz-type boundary value problem can be constructed in the associated Hardy classes.  
\end{abstract}

\maketitle

\section{Introduction}

In this paper, we work to illustrate connections between Hardy classes of functions on the unit disk which satisfy certain nonhomogeneous Cauchy-Riemann equations and the Hardy spaces of distributions on the unit circle 

In \cite{GHJH2}, G. Hoepfner and J. Hounie prove that functions in the holomorphic Hardy spaces of functions on the unit disk have boundary values in the sense of distributions and these distributions can be represented by an atomic decomposition. While Coifman proved the analogous result in \cite{Coif1} for the upper half space of $\mathbb{R}^2$ and described the unit disk case in \cite{Coif2}, Hoepfner and Hounie filled a gap in the literature in making this proof for the disk (except the $p = 1$ case which can be found in \cite{Koosis}). One aspect of the result in \cite{GHJH2}, Theorem 2.2, of interest is that the theorem is not stated as biconditional, as it is in \cite{Coif1}. From L. Colzani's thesis \cite{Col}, we know that being a distribution on the unit circle with an atomic decomposition is equivalent to that distribution's Poisson integral having finite radial maximal function in the $L^p$ norm, but Colzani's result does not address whether these distributions are boundary values in the sense of distributions of certain classes of functions defined on the interior of the disk. Later in \cite{GHJH2}, the appropriate biconditional statement, with respect to functions in the holomorphic Hardy spaces, is given as Theorem 6.2. Yet, the following question arises. What classes of functions, other than the holomorphic Hardy space functions, have boundary values in the sense of distributions on the circle that have an atomic decomposition?

We consider Hardy classes of functions modeled on the holomorphic Hardy spaces by exchanging the condition of holomorphicity for being a solution to a nonhomogeneous Cauchy-Riemann equation $\frac{\p w}{\p\z} = f$, where $f \in L^q(D)$, $q>2$. This class of functions is represented by a sum of a holomorphic Hardy space function and Vekua's $T$ operator \cite{Vek} applied to the nonhomogeneous part $f$ of the corresponding partial differential equation. Consequently, boundary values in the sense of distributions of functions in these generalized Hardy classes are representable as the sum of a boundary value in the sense of distributions of a holomorphic Hardy space function and an error term, where the error term is manageable. Thus, these distributions are representable as a sum of an atomic decomposition and a nice error term. This allows us to generalize the classes of functions which have boundary values in the sense of distributions with an atomic representation to classes of functions whose boundary values in the sense of distributions are almost representable by an atomic decomposition. 

Using the found representation, we extend one of the applications from \cite{GHJH2}. In that paper, the authors use the atomic decomposition result to prove that the Hilbert transform on the circle is a continuous operator on the space of boundary values in the sense of distributions of functions in the holomorphic Hardy space. 
This application is nontrivial as the Hilbert transform is not a continuous operator on $L^p$ of the circle for any $p \leq 1$. We extend this result to the distributions which are boundary values of the Hardy classes of solutions to the specified nonhomogeneous Cauchy-Riemann equations. This provides examples of additional small $p$ classes where the Hilbert transform applied to the associated space of boundary values is an improvement over the associated small $p$ Lebesgue spaces. Also, we define a variation of the Schwarz boundary value problem such that the Hardy classes we have defined contain constructable solutions to this problem. This shows a direct connection between our generalized Hardy classes on the disk and the Hardy spaces of distributions on the circle because, in this boundary value problem, the boundary condition is a Hardy space distribution. This generalizes the Schwarz problem to a more general boundary condition, e.g., the boundary condition is a continuous function in \cite{BegBook} and \cite{Beg}. Additionally, this construction is an alternative way of extending atomic distributions to the interior of the disk since the constructed function is not necessarily harmonic. 

The arguments of the proofs of these results are general enough that they extend directly, for a restricted $p$ range, to the analogous cases for arbitrarily higher-order nonhomogeneous Cauchy-Riemann equations. This provides further example classes of functions with boundary values in the sense of distributions that are close in a qualitative way to having an atomic decomposition. Also, not only does this method provide an extension of atomic distributions to a solution of a first-order nonhomogeneous partial differential equation, it extends these distributions to solutions of higher-order equations, as we can constructively solve the associated $n^{\text{th}}$-order Schwarz problem with $n$ boundary conditions which are all Hardy space distributions.

We describe the outline of the article. In Section 2, we define the Hardy classes that we consider here, relate this more general definition to the classical Hardy spaces, and collect results about these classes. In Section 3, we show that the Hilbert transform is a continuous operator on the space of boundary values in the sense of distributions of these classes. Also, we define the Schwarz-type boundary value problem and show how to explicitly construct solutions that are members of the Hardy classes defined in Section 2. In Section 4, we define Hardy classes of solutions to higher-order nonhomogeneous Cauchy-Riemann equations and show that the results of Sections 2 and 3, found initially in the first-order setting, extend to these Hardy classes by iteration of the first-order results.

We thank Professor Gustavo Hoepfner for the suggestion to consider an ``almost'' atomic decomposition. We also thank Professor Andrew Raich for the countless helpful discussions, being generous with his time and suggestions, and his overall mentorship during the time when this work was being produced.  

\section{Definitions and Results}

Throughout we will let $D(0,r)$ denote the disk centered at the origin with radius $r$ in the complex plane, $D := D(0,1)$  is the unit disk in the complex plane, $\p D(0,r)$  and $\p D$ denote the boundary of $D(0,r)$ and $D$ respectively, $C^{k,\alpha}(D)$ is the space of functions on $D$ that are differentiable up to order $k$ with $k^{\text{th}}$ derivative in the space of $\alpha$-H\"older continuous functions, $C^\infty(\p D)$ is the space of smooth functions defined on $\p D$, and $\dc$ is the space of distributions on $\p D$.

\begin{deff}\label{hardyspacedef}
For each $p$ in the range $0 < p < \infty$ and function $f: D \to \mathbb{C}$, we define $H^p_{f}(D)$ to be the set of functions $w : D \to \mathbb{C}$ that are solutions to
\begin{equation*}
\frac{\p w}{\p \z} = f, \label{CR}
\end{equation*}
and satisfy
\[
||w||_{\Har} := \left(\sup_{0< r< 1} \frac{1}{2\pi} \int_{0}^{2\pi} |w(re^{i\theta})|^p \,d\theta\right)^{1/p} < \infty.
\]
\end{deff}

The case $f \equiv 0$ is that of the classic holomorphic Hardy spaces, which we denote by $\Har$. For background on the spaces $\Har$, see \cite{Duren}, \cite{Koosis}, and \cite{Rep}. We consider here only the cases $0 < p \leq 1$.

We define a boundary value in the sense of distributions as in \cite{GHJH2}. 

\begin{deff}Let $f$ be a function defined on $D$. We say that $f$ has a boundary value in the sense of distributions, denoted by $f_b$, if, for every $ \varphi \in C^\infty(\partial D)$, the limit
            \[
             \langle f_b, \varphi \rangle := \lim_{r \nearrow 1} \int_0^{2\pi} f(re^{i\theta}) \, \varphi(\theta) \,d\theta
            \]
            exists.
            
\end{deff}

From Theorem 2.2, Theorem 3.1, and Corollary 3.1 of \cite{GHJH2}, we know that the functions $w \in \Har$ have boundary values in the sense of distributions $w_b$ and those distributions can be represented as
\begin{equation}\label{atomicdecomp}
w_b = \sum_{n} c_n a_n,
\end{equation}
where $\{c_n\} \in \ell^p$ and $\{a_n\}$ is a collection of functions defined on $\p D$ satisfying 
\begin{itemize}
\item $\supp(a_n) \subset J_n$, where $J_n$ is an arc in $\p D$ (that could be all of $\p D$),
\item $|a_n(\theta)|\leq |J_n|^{-1/p}$,
\item $\int_0^{2\pi} a_n(\theta) \, \theta^k \,d\theta = 0,$ for $k \leq \frac{1}{p} -1$,
\end{itemize}
or $a_n \equiv c$ for some constant $c$ such that $|c| \leq 1$. We will call such $a_n$ $p$-atoms, and we will call a representation of the form of (\ref{atomicdecomp}) an atomic decomposition.

Now, we state auxiliary results that indicate some natural restrictions to place on the classes $\Hp$ that we consider. 

\begin{theorem}[Theorem 1.16 \cite{Vek} p.34]\label{oneonesix}
If $\frac{\p w}{\p\z} = f \in L^1(D)$, then 
\begin{equation}\label{cp}
w(z) = \Phi(z) - \frac{1}{\pi} \iint_D \frac{f(\zeta)}{\zeta - z}\,d\xi\,d\eta,
\end{equation}
where $\zeta = \xi + i\eta$ and $\Phi$ is holomorphic in $D$. 
\end{theorem}

In \cite{Vek}, this type of formula is called the ``representation of the second kind,'' and we adopt this convention. The similarity between this representation and the well-known Cauchy-Pompeiu formula is clear. In \cite{KlimBook}, it is shown that in the cases of $w \in \Hp$ with:
\begin{enumerate}
\item[$(i)$] $1 \leq p < \infty$, $f = Aw + B\bar{w}$, $A, B \in L^s(D)$, $s>2$, or 
\item[$(ii)$] $1 < p < \infty$, $f = q\frac{\p w}{\p z}$, $q \in C^{k,\alpha}(D)$, $k\geq 0, 0 < \alpha < 1$, $|q(z)|\leq C(1-|\zeta(z)|)^{2\epsilon}$, where $0 < \epsilon < 1$, and $\zeta$ is the principle homeomorphism of the corresponding Beltrami equation,
\end{enumerate}
that $\Phi$ in (\ref{cp}) is in $\Har$. For more information about the specific equations involved in item $(i)$, see \cite{Vek} and \cite{KlimBook}, and for more information about the equations involved in item $(ii)$, see \cite{Vek} and \cite{ellipquasi}. To be clear, in Definition \ref{hardyspacedef}, we choose a specific partial differential equation $\frac{\p w}{\p\z} = f$, then the class $H^p_{f}(D)$ is the subset of the solutions to that equation which also satisfy the integrability condition associated with the holomorphic Hardy spaces of functions. Since the results that follow hold for any class of functions that satisfy the conditions of the definition, this allows us to study the solutions to a wide class of partial differential equations, including nonlinear partial differential equations. We work to show that a representation of the second kind exists in the more general setting. Let us state a few more results that we use to control the integral part in the representation.

\begin{deff}
For $f: D \to \mathbb{C}$ and $z  \in \mathbb{C}$, we denote by $T(\cdot)$ the integral operator defined by 
\[
T(f)(z) = -\frac{1}{\pi} \iint_D \frac{f(\zeta)}{\zeta - z}\, d\xi\,d\eta,
\]
whenever the integral is defined.
\end{deff}

\begin{theorem}[Theorem 1.26 \cite{Vek} p.\,47] \label{Vekonepointtwentysix}
If $f \in L^q(D)$, $1 \leq q \leq 2$, then $T(f) \in L^\gamma(D)$, $1 < \gamma < \frac{2q}{2-q}$. 
\end{theorem}

\begin{theorem}[Theorem 1.4.7 \cite{KlimBook} p.\,17]\label{onefourseven}
If $f \in L^q(D)$, $1 < q \leq 2$, then $T(f)|_{\p D(0,r)} \in L^\gamma(\p D(0,r))$, $0 < r \leq 1$, where $\gamma$ satisfies $1 < \gamma < \frac{q}{2-q}$, and 
\[
||Tf||_{L^\gamma(\p D(0,r))} \leq C ||f||_{L^q(D)},
\]
where the constant does not depend on $r$ or $f$. 
\end{theorem}

\begin{theorem}[Theorem 1.4.8 \cite{KlimBook} p.\,18]\label{onefoureight}
If $f \in L^q(D)$, $1 < q \leq 2$, then 
\[
\lim_{r\nearrow 1} \int_0^{2\pi} |T(f)(e^{i\theta}) - T(f)(re^{i\theta})|^\gamma \,d\theta = 0, 
\]
for $1 \leq \gamma < \frac{q}{2-q}$. 
\end{theorem}

\begin{theorem}[Theorem 1.19 \cite{Vek} p.38]\label{onepointnineteen}
For $f \in L^q(D)$, $q > 2$, $T(f) \in C^{0,\alpha}(\overline{D})$, with $\alpha = \frac{ q-2}{q}$. 
\end{theorem}

Let us define the classes that we will now consider.

\begin{deff}
For $0 < p \leq 1$, we denote by $\Hpq$ the space of functions $w \in \Hp$ where $f \in L^q(D)$, $q > 1$.
\end{deff}

Our main result is the following.

\begin{theorem}\label{main}
For $0 < p \leq 1$ and $q>1$, every $w \in \Hpq$ has a representation of the second kind
\[
w(z) = \Phi(z) + T(f)(z),
\]
where $\Phi \in \Har$ and 
\[
T(f)(z) := -\frac{1}{\pi} \iint_D \frac{f(\zeta)}{\zeta - z} \,d\xi\,d\eta.
\]
Also, the boundary value in the sense of distributions of $w$, denoted by $w_b$, exists and can be represented as 
\[
w_b = \sum_{n} c_n a_n + T(f)_b.
\]
where $\{c_n\} \in \ell^p$, $\{a_n\}$ is a collection of $p$-atoms, and 
\[
T(f)_b \in \begin{cases} L^\gamma(\p D),\quad 1 < \gamma < \frac{q}{2-q}, & \text{ when } 1 < q \leq 2 \\
C^{0,\alpha}(\p D), \quad \alpha = \frac{q - 2}{q}, & \text{ when } q >2.
\end{cases}
\]
\end{theorem}

\begin{proof}
By assumption, $\frac{\p w}{\p\z} = f \in L^q(D) \subset L^1(D)$, $1 < q  \leq 2$. So, by Theorem \ref{oneonesix}, $w$ has a representation of the second kind
\begin{equation}\label{repofsecondkind}
w(z) = \Phi(z) + T(f)(z),
\end{equation}
where $\Phi$ is holomorphic in $D$. Observe that
\begin{align*}
&\int_0^{2\pi} |\Phi(re^{i\theta})|^p \,d\theta  \\
&= \int_0^{2\pi} |w(re^{i\theta}) - T(f)(re^{i\theta})|^p \,d\theta  \\
&\leq \int_0^{2\pi} |w(re^{i\theta})|^p \,d\theta + \int_0^{2\pi} | T(f)(re^{i\theta})|^p \,d\theta\\
&\leq ||w||^p_{H^p(D)} + \int_0^{2\pi} | T(f)(re^{i\theta})|^p \,d\theta.
\end{align*}
For each $r$, let $D_{r,*} = \{\theta \in [0,2\pi] : |T(f)(re^{i\theta})| \leq 1 \}$ and $D^{r,*} = \{\theta \in [0,2\pi] : |T(f)(re^{i\theta})| > 1 \}$. By Theorem \ref{onefourseven}, since $f \in L^q(D), 1 < q \leq 2$, it follows that 
\[
\sup_{0 < r < 1}||T(f)||_{L^\gamma(\p D(0,r))} \leq C ||f||_{L^q(D)},
\]
where $\gamma$ satisfies $1 < \gamma < \frac{q}{2-q}$ and $C$ is a constant that does not depend on $f$. So
\begin{align*}
&\int_0^{2\pi} | T(f)(re^{i\theta})|^p \,d\theta \\
&= \int_{D_{r,*}} | T(f)(re^{i\theta})|^p \,d\theta + \int_{D^{r,*}} | T(f)(re^{i\theta})|^p \,d\theta \\
&\leq 2\pi + \int_0^{2\pi} | T(f)(re^{i\theta})|^\gamma \,d\theta \\
&\leq 2\pi + C||f||_{L^q(D)} < \infty.
\end{align*}
Hence, 
\begin{align*}
\int_0^{2\pi} |\Phi(re^{i\theta})|^p \,d\theta 
& \leq  ||w||^p_{H^p(D)} + \int_0^{2\pi} | T(f)(re^{i\theta})|^p \,d\theta \\
&\leq  ||w||^p_{H^p(D)} +  2\pi + C||f||_{L^q(D)} < \infty,
\end{align*}
where the right hand side has no dependence on $r$. Thus, 
\[
\sup_{0 < r < 1} \int_0^{2\pi} |\Phi(re^{i\theta})|^p \,d\theta \leq ||w||^p_{H^p(D)} +  2\pi + C||f||_{L^q(D)} < \infty,
\]
and $\Phi \in H^p(D)$. 

Now, using the representation of the second kind (\ref{repofsecondkind}), for any $\varphi \in C^\infty(\p D)$, we have
\begin{align*}
\left| \langle w_b, \varphi \rangle \right| &= \left| \lim_{r \nearrow 1} \int_0^{2\pi} w(re^{i\theta}) \varphi(\theta) \,d\theta \right|\\
&= \left| \lim_{r \nearrow 1} \int_0^{2\pi} (\Phi(re^{i\theta}) + T(f)(re^{i\theta})) \varphi(\theta) \,d\theta \right| \\
&= \left| \lim_{r \nearrow 1} \int_0^{2\pi} \Phi(re^{i\theta})\,\varphi(\theta) \,d\theta + \lim_{r\nearrow 1} \int_0^{2\pi} T(f)(re^{i\theta})\, \varphi(\theta) \,d\theta \right| \\
&\leq \left| \lim_{r \nearrow 1} \int_0^{2\pi} \Phi(re^{i\theta})\,\varphi(\theta) \,d\theta \right| + \left| \lim_{r\nearrow 1} \int_0^{2\pi} T(f)(re^{i\theta})\, \varphi(\theta) \,d\theta \right| \\
&\leq \left| \lim_{r \nearrow 1} \int_0^{2\pi} \Phi(re^{i\theta})\,\varphi(\theta) \,d\theta \right| + C_\varphi  \lim_{r\nearrow 1} \int_0^{2\pi} \left|T(f)(re^{i\theta})\right| \,d\theta ,
\end{align*}
where $C_\varphi$ is a constant that bounds $|\varphi(\theta)|$ for all $\theta\in [0,2\pi]$. By Theorem 3.1 from \cite{GHJH2}, since $\Phi \in \Har$, it follows that the first summand on the right hand side is finite, and by Theorem \ref{onefourseven}, the second summand is finite. Hence, $w_b$, the boundary value in the sense of distributions of $w$, exists, and 
\[
w_b = \Phi_b + T(f)_b,
\]
where $\Phi_b$ and $T(f)_b$ are the boundary values in the sense of distributions of $\Phi$ and $T(f)$, respectively.

 By Theorem 2.2 from \cite{GHJH2}, there exists $\{c_n\} \in \ell^p$ and a collection $\{a_n\}$ of $p$-atoms such that
\[
\Phi_b = \sum_{n} c_n a_n.
\]
Observe that by Theorem \ref{onefoureight}, if $T(f)_b$ denotes the boundary value in the sense of distributions of $T(f)(z)$ and $T(f)(e^{i\theta})$ denotes the $L^\gamma(\p D)$ boundary value, we have, for any $\varphi \in C^\infty(\p D)$, 
\begin{align*}
\left|\langle T(f)_b, \varphi \rangle - \langle T(f)(\cdot), \varphi \rangle \right|
&= \left| \lim_{r \nearrow 1} \int_0^{2\pi} T(f)(re^{i\theta})\varphi(\theta) \,d\theta - \int_0^{2\pi} T(f)(e^{i\theta}) \varphi(\theta) \,d\theta \right| \\
&\leq C_\varphi \lim_{r\nearrow 1} \int_0^{2\pi} |T(f)(re^{i\theta}) - T(f)(e^{i\theta})| \,d\theta = 0,
\end{align*}
where $C_\varphi$ is constant that bounds $|\varphi(\theta)|$ for all $\theta \in [0,2\pi]$. So, $T(f)_b$ and $T(f)(e^{i\theta})$ agree as distributions. Therefore, 
\[
w_b = \sum_{n} c_n a_n + T(f)_b,
\]
where ${c_n} \in \ell^p$, $\{a_n\}$ are a collection of $p$-atoms, and $T(f)_b \in L^\gamma(\p D)$, $1 < \gamma< \frac{q}{2-q}$. If $q > 2$, then, by Theorem \ref{onepointnineteen}, $T(f) \in C^{0,\alpha}(\overline{D})$, $\alpha = \frac{ q-2}{q}$. Hence, $(T(f))_b \in C^{0,\alpha}(\p D)$. 
\end{proof}

We now consider the specific case of $f = Aw+ B\bar{w}$ and $\frac{1}{2} < p \leq 1$ as an example of the classes $\Hpq$. 

\begin{examp}
Let $A, B \in L^s(D)$, where $s >2$. By Theorem 2.1.5 in \cite{Klim2016}, if $w \in \Hp$, for  $p > 0$, then $w \in L^m(D)$, for all $m$ such that $0 < m < 2p$. So, if $p$ is restricted to $\frac{1}{2} < p \leq 1$, then $2p > 1$. Consider, for some $1 < q \leq 2$, 
\begin{align*}
\iint_D |(Aw+B\bar{w})(z)|^q \,dx\,dy
&\leq 2^{q-1}\iint_D |(Aw)(z)|^q \,dx\,dy + 2^{q-1}\iint_D |(B\bar{w})(z)|^q \,dx\,dy.
\end{align*}
Now, we appeal to H\"older's inequality. We choose the appropriate $\alpha$ and $\beta$ such that $\frac{1}{\alpha} + \frac{1}{\beta} = 1$ with respect to $q$ in the range $1 < q \leq 2$ so that $A,B \in L^{\alpha q}(D)$ and $w \in L^{\beta a}(D)$. For $q$ sufficiently close to 1, we choose $\beta >1$ such that $1 < q < \beta q < 2p$, as $2p > 1$ here. For such a $\beta$, there is a corresponding $\alpha$ and 
\begin{align*}
&2^{q-1}\iint_D |(Aw)(z)|^q \,dx\,dy + 2^{q-1}\iint_D |(B\bar{w})(z)|^q \,dx\,dy \\
&\leq 2^{q-1}\left(\iint_D |A(z)|^{\alpha q} \,dx\,dy\right)^{1/\alpha} \left(\iint_D |w(z)|^{\beta q} \,dx\,dy\right)^{1/\beta} \\
&\quad\quad + 2^{q-1}\left(\iint_D |B(z)|^{\alpha q} \,dx\,dy\right)^{1/\alpha} \left(\iint_D |w(z)|^{\beta q} \,dx\,dy\right)^{1/\beta}.
\end{align*}
Since $\beta q< 2p$, it follows that $w \in L^{\beta q}(D)$. Now, choosing the $A,B \in L^{\alpha q}(D)$, the right hand side of the above inequality is finite, and so $Aw+B\bar{w} \in L^q(D)$. We are assuming $A, B \in L^s(D)$, for some $s>2$, so an improvement is needed only in the case that $\alpha q > s$.
\end{examp}

There are other integral operators that could be used in place of $T(\cdot)$. We describe one such operator and collect some results concerning it that are used in Section \ref{SchwarzSection}. 

\begin{deff}
For $f: \mathbb{C} \to \mathbb{C}$ and $z \in \mathbb{C}$, we denote by $\widetilde{T}(\cdot)$ the integral operator defined by 
\[
\widetilde{T}(f)(z) = -\frac{1}{\pi} \iint_D \left(\frac{f(\zeta)}{\zeta - z} + \frac{z\overline{f(\zeta)}}{1-\overline{\zeta}z} \right)\, d\xi\,d\eta,
\]
whenever the integral is defined.
\end{deff}

Similarly to $T(\cdot)$, $\widetilde{T}(f)$ is defined for any $f \in L^1(D)$, see Theorem 1.4.2 of \cite{KlimBook} (where the result is cited from \cite{Vek}). Also from p.16 of \cite{KlimBook}, the operator $\widetilde{T}(\cdot)$ shares the property of being a right inverse to the Cauchy-Riemann operator, which is well-known for the operator $T(\cdot)$ (see \cite{Vek}), i.e., 
\[
\frac{\p}{\p\z} \widetilde{T}(f) = f,
\]
whenever $\widetilde{T}(f)$ is defined. Note that Theorems \ref{Vekonepointtwentysix} and \ref{onefourseven} also hold for $\widetilde{T}(\cdot)$, see their cited references. This leads to the following alternative to Theorem \ref{main}.

\begin{theorem}\label{altmain}
Every $w \in \Hpq$ can be represented as 
\[
w(z) = \Phi(z) + \widetilde{T}(f)(z),
\]
where $\Phi \in \Har$.
Also, the boundary value in the sense of distributions of $w$, denoted by $w_b$, exists and can be represented as 
\[
w_b = \Phi_b + (\widetilde{T}(f))_b.
\]
where $\Phi_b$ is the boundary value in the sense of distributions of $\Phi$ and $(\widetilde{T}(f))_b$ is the boundary value in the sense of distributions of $\widetilde{T}(f)$. 
\end{theorem}

\begin{proof}
Observe that 
\[
\frac{\p}{\p\z} (w - \widetilde{T}(f)) = 0,
\]
so 
\[
\Phi := w - \widetilde{T}(f)
\]
is holomorphic. The justification that this $\Phi \in \Har$ is precisely the same size estimate argument as the analogous statement in the proof of Theorem \ref{main}, and the justification of the statement about boundary values in the sense of distributions is similarly the same argument as in the proof of the analogous result in Theorem \ref{main}.
\end{proof}

The reason for establishing the alternative representation of the second kind is to appeal to the fact that
\[
\frac{f(\zeta)}{\zeta - z} + \frac{z\overline{f(\zeta)}}{1-\overline{\zeta}z} = \frac{f(\zeta)}{\zeta - z} - \overline{\left( \frac{f(\zeta)}{\zeta - z} \right)} = 2i \im\left\{\frac{f(\zeta)}{\zeta - z} \right\}
\]
for $|z| = 1$, so 
\[
\re\{(\widetilde{T}(f))_b\} = 0.
\]
In Section \ref{SchwarzSection}, this representation is used to construct solutions for the boundary value problem considered there.

\section{Applications}

\subsection{Hilbert Transform}

We work to show an application of Theorem \ref{main}.
As a combination of Theorem 6.2 and Corollary 6.2 from \cite{GHJH2}, we have the following.

\begin{deff}
Let $H^p_{at}(\p D)$ denote the $f \in \dc$ where there exists $\{c_n\}\in \ell^p$ and a collection of $p$-atoms $\{a_n\}$ such that 
\[
f = \sum_{n} c_n a_n
\]
as distributions.
\end{deff}

\begin{deff}
For any $f \in L^1(\p D)$, the Hilbert transform $H(f)$ is given by 
\[
H(f)(\theta) = \lim_{\epsilon \to 0} \frac{1}{\pi} \int_{\epsilon \leq |t|\leq \pi} \frac{f(\theta - t)}{2\tan(t/2)} \,dt.
\]
\end{deff}

\begin{prop}\label{hilbatomiccont}
The Hilbert transform applies $H^p_{at}(\p D)$ continuously to $H^p_{at}(\p D)$.
\end{prop}

From \cite{Rep} p.137, we have the following theorem attributed to Marcel Riesz.

\begin{theorem}\label{hilblpcont}
Let $u \in L^p(\p D)$, $1 < p < \infty$. The Hilbert transform of $u$, denoted by $H(u)$, is in $L^p(\p D)$, and 
\[
||H(u)||_{L^p(\p D)} \leq C ||u||_{L^p(\p D)},
\]
where $C$ is a constant that depends only on $p$. 
\end{theorem}

Let us define the space that we will now be considering.

\begin{deff}
Denote by $(\Hpq)_b$ the set of distributions $w_b \in \mathcal{D}'(\p D)$ that are boundary values in the sense of distributions of $w \in \Hpq$. 
\end{deff}

Let us define a quasi-norm on our classes $(\Hpq)_b$.

\begin{deff}
We define the quasi-norm $||\cdot||$ by
\[
||w_b|| := ||\sum_{n}c_n a_n ||_{at} + ||T(f)_b||_{L^\gamma(\p D)},
\]
for each $w_b \in (\Hpq)_b$, $0 < p \leq 1$, where $f \in L^q(D), q>1$, and 
\[
w_b = \sum_{n} c_n a_n + T(f)_b
\]
is the representation from Theorem \ref{main}, $\gamma$ is the same number as in the statement of Theorem \ref{onefourseven}, and
\[
||\sum_{n} c_n a_n||_{at} := \left( \inf \sum_{n} |c_n|^p \right)^{1/p},
\]
where the infimum is taken over the set of all sequences of coefficients $\{c_n\}$ for equivalent atomic decompositions. 
\end{deff}

Note that by its definition, $||\cdot||$ is a quasi-norm, and we take the statements in Proposition \ref{hilbatomiccont} and Theorem \ref{hilblpcont} to be continuity of the Hilbert transform as an operator on the corresponding quasi-normed classes (of course since $\gamma>1$ here Theorem \ref{hilblpcont} is in terms of an actual norm).

As a consequence of Proposition \ref{hilbatomiccont}, Theorem \ref{hilblpcont}, and the representation we have found for boundary values in the sense of distributions of functions in $\Hpq$ when $0 < p \leq 1$ and $f \in L^q(D)$, $q > 1$, from Theorem \ref{main}, we have the following result concerning the Hilbert transform.

\begin{theorem}
For $\Hpq$, where $0 < p \leq 1$ and $f \in L^q(D)$, $q>1$, the Hilbert transform is a continuous operator on the space $(\Hpq)_b$.
\end{theorem}

\begin{proof}
Let $w \in \Hpq$, for $0 < p \leq 1$ and $f \in L^q(D)$ where $q>1$. By Theorem \ref{main}, this $w$ has a boundary value in the sense of distributions 
\[
w_b = \sum_{n} c_n a_n + T(f)_b,
\]
where $\{c_n\} \in \ell^p$, $\{a_n\}$ is a collection of $p$-atoms and $T(f)_b \in L^\gamma(\p D)$ for $\gamma$ that satisfies $1 < \gamma < \frac{q}{2-q}$. By Proposition \ref{hilbatomiccont}, the Hilbert transform is continuous from $H^p_{at}(\p D)$ to itself. By Theorem \ref{hilblpcont}, the Hilbert transform is continuous from $L^\gamma(\p D)$ to $L^\gamma(\p D)$, so long as $1 < \gamma< \infty$. Since $w_b  = \sum_{n} c_n a_n + T(f)_b$ and the Hilbert transform is a linear operator, it follows that, denoting the Hilbert transform by $H$, we have
\begin{align*}
||H(w_b)||  
&= ||H(\sum_{n} c_n a_n + T(f)_b) || \\
&= ||H(\sum_{n} c_n a_n) + H(T(f)_b)|| \\
&= ||H(\sum_{n} c_n a_n)||_{at} + ||H(T(f)_b)||_{L^\gamma(\p D)} \\
&\leq C_p||\sum_{n} c_n a_n||_{at} + C_\gamma||T(f)_b||_{L^\gamma(\p D)} \\
&\leq C\left(||\sum_{n} c_n a_n||_{at} + ||T(f)_b||_{L^\gamma(\p D)} \right) \\
&= C\left(||\sum_{n} c_n a_n + T(f)_b|| \right)\\
&= C||w_b||,
\end{align*}
where $C_p$ and $C_\gamma$ are constants that depend only on $p$ and $\gamma$ respectively and $C:= \max\{C_p, C_\gamma\}$. 
\end{proof}

Denote by $H^p_{at}(\p D) + L^q(\p D)$ the set 
\[
H^p_{at}(\p D) + L^q(\p D) := \{f = g + h : g \in H^p_{at}(\p D) \quad\text{and}\quad h \in L^q(\p D)\},
\]
for some $q$ satisfying $1 < q < \infty$ with the quasi-norm $||\cdot||$ defined by
\[
||f|| = ||g||_{at} + ||h||_{L^q(\p D)}.
\] 

\begin{corr}
The Hilbert transform maps $H^p_{at}(\p D) + L^q(\p D)$ continuously to  $H^p_{at}(\p D) + L^q(\p D)$. 
\end{corr}

\subsection{Schwarz Boundary Value Problem}\label{SchwarzSection}

First, we recall the real-variable Hardy spaces $\mathcal{H}^p( D)$ analogous to our setup, for more information see \cite{Col} and \cite{BigStein}.

\begin{deff}
We define $\mathcal{H}^p( D)$, $0 < p \leq 1$, to be the space of distributions $g \in \mathcal{D}'(\p D)$ such that 
\[
\int_0^{2\pi} \left|\sup_{0 \leq r < 1} \left| \frac{1}{2\pi} \langle g, P_r(\theta - \cdot) \rangle \right|\right|^p\, d\theta < \infty,
\]
where 
\[
P_r(\theta) = \frac{1-r^2}{1- 2r\cos(\theta) + r^2}
\]
is the Poisson kernel on $D$. 
\end{deff}

Next, we identify a specific subset of $\mathcal{H}^p(D)$. 

\begin{deff}
We define $\re{(H^p(D))_b}$ to be the space of real parts of the boundary values in the sense of distributions of functions in $H^p(D)$, i.e., $h \in \re{(H^p(D))_b}$ if there exists $H \in H^p(D)$ such that $\re{H_b} = h$. 
\end{deff}

We introduce the following two results. 

\begin{theorem}[Theorem 6.2 \cite{GHJH2}]\label{GHJH2sixpointtwo}
Let $h \in \mathcal{D}'(\p D)$. The following are equivalent:
\begin{enumerate}
\item $h \in \re{(H^p(D))_b} + i \re{(H^p(D))_b} := \{ g_1 + ig_2: g_1,g_2 \in \re{(H^p(D))_b}\}$,
\item $\langle h, P_r(\theta - \cdot) + i Q_r(\theta - \cdot) \rangle \in H^p(D)$, where 
\[
Q_r(\theta) = \frac{2r\sin(\theta)}{1-2r\cos(\theta) + r^2}
\]
is the conjugate Poisson kernel, and 
\item there exists a sequence of $p$-atoms $\{a_n\}$ and a sequence $\{c_n\} \in \ell^p$ such that 
\[
h = \sum_{n} c_n a_n.
\]
\end{enumerate}
\end{theorem}

\begin{prop}[Proposition 2 \cite{Col}]\label{Colzanitwo}
For every sequence $\{a_n\}$ of $p$-atoms and sequence $\{c_n\} \in \ell^p$, there is a distribution $h \in \mathcal{H}^p(D)$ such that 
\[
h = \sum_n c_n a_n.
\] 
\end{prop}

\begin{prop}\label{Hardysubset}
The space of distributions $\re{(H^p(D))_b}$ is a subset of $\mathcal{H}^p(D)$. 
\end{prop}

\begin{proof}
Let $h \in \re{(H^p(D))_b}$. By definition, $ h = h + i 0 \in \re{(H^p(D))_b} + i \re{(H^p(D))_b}$. By Theorem \ref{GHJH2sixpointtwo}, there exist sequences $\{a_n\}$ of $p$-atoms and $\{c_n\} \in \ell^p$ such that $h = \sum_n c_n a_n$. By Proposition \ref{Colzanitwo}, $h = \sum_n c_n a_n \in \mathcal{H}^p(D)$. 
\end{proof}

Next, we describe a specific variation of the Schwarz boundary value problem that is natural to consider with the space of functions $\Hpq$.

\begin{deff}
We say a function $w: D \to \mathbb{C}$ solves the Schwarz-type boundary value problem $\mathcal{S}_{f,h}$ whenever $w$ satisfies the following:
\begin{enumerate}
\item $\frac{\p w}{\p\z} = f$, for some function $f: D \to \mathbb{C}$ and  
\item $\re{w_b} = h$ as distributions, where $h \in \mathcal{D}'(\p D)$.
\end{enumerate}
\end{deff}

This problem is explicitly solved by Theorem 2.1 of \cite{Beg} for the case $f \in L^1(D)$ and $h \in C(\p D)$. In the proof of the next theorem, we construct solutions to a variation of the problem which are in the Hardy classes $\Hpq$.

\begin{theorem}\label{schwarzfirst}
For each $f \in L^q(D)$, $q>1$, and $h \in \re{(H^p(D))_b} $, $0 < p \leq 1$, there exist $w \in \Hpq$ which solve the boundary value problem $\mathcal{S}_{f,h}$. 
\end{theorem}

\begin{proof}
We fix arbitrary $f \in L^q(D)$, for some $q$ satisfying $q>1$, and $h \in \re{(H^p(D))_b}$, for some $p$ satisfying $0 < p \leq 1$. We explicitly describe the construction of $w \in \Hpq$ that satisfies $\mathcal{S}_{f,h}$.

From Theorem \ref{GHJH2sixpointtwo}, 
\[
H = \langle h, P_r(\theta - \cdot) + i Q_r(\theta - \cdot) \rangle \in H^p(D).
\]
Let 
\[
w = H + \widetilde{T}(f).
\]
By its construction, $w \in \Hpq$. Observe that, for $|z| = 1$,  
\begin{align*}
\widetilde{T}(f)(z) 
&= -\iint_D \left( \frac{f(\zeta)}{\zeta - z} + \frac{ z\overline{f(\zeta)}}{1-z\overline{\zeta}}\right) d\xi \,d\eta \\
&= -\iint_D \left( \frac{f(\zeta)}{\zeta - z} - \overline{\left(\frac{ f(\zeta)}{\zeta-z}\right)}\right) d\xi \,d\eta ,
\end{align*}
which is purely imaginary. By Theorem 3.1 of \cite{GHJH2}, $\re\{H_b\} = h$. Hence, 
\begin{align*}
\re\{w_b\}
&= \re\{(H + \widetilde{T}(f))_b\}\\
&= \re\{H_b\} + \re\{(\widetilde{T}(f))_b\} \\
&= h.
\end{align*}

\end{proof}

\begin{corr}\label{atomicdecompforfirst}
For any $w\in \Hpq$ that solves $\mathcal{S}_{f,h}$ for some $f \in L^q$, $q>1$ and $h \in \re{(H^p(D))_b}$, $0 < p  \leq 1$, with representation $w = \langle h, P_r(\theta - \cdot) + i Q_r(\theta - \cdot) \rangle + \widetilde{T}(f)$ such that $(\langle h, iQ_r(\theta - \cdot) + \widetilde{T}(f))_b= 0$, $w_b$ has an atomic decomposition. 
\end{corr}

\begin{proof}
Reusing all objects from the proof of the theorem, since 
\[
w_b = H_b + (\widetilde{T}(f))_b
\]
and 
\[
\re\{w_b\} = h, 
\]
it follows that if $(\langle h, iQ_r(\theta - \cdot) + \widetilde{T}(f))_b= 0$, then 
\begin{align*}
w_b &= \re\{w_b\} + i \im\{w_b\} = h + i \im\{(\langle h, iQ_r(\theta - \cdot) + \widetilde{T}(f))_b\}= h.
\end{align*}
By Theorem \ref{GHJH2sixpointtwo}, $h$ has an atomic decomposition
\[
h = \sum_{n} c_n a_n.
\]
Thus, 
\[
w_b = \sum_{n} c_n a_n.
\]
\end{proof}

Note that the above result gives us conditions which characterize a subset of $\Hpq$ having boundary values in the sense of distributions with a representation by an atomic decomposition.

\section{Extension to higher-order equations}

In this section, we describe how the methods that we applied above in the first-order case extend to analogous results for Hardy classes of solutions to higher-order equations by iteration.

First, let us define the higher-order variants of the classes $\Hp$.

\begin{deff}For each positive integer $n \geq 1$ and real $p$ in the range $0 < p < \infty$, we define $H^{n,p}_{f}(D)$ to be the set of functions $w : D \to \mathbb{C}$ that are solutions to
\begin{equation*}
\frac{\p^n w}{\p \z^n} = f, \label{CRhigher}
\end{equation*}
where $f: D \to \mathbb{C}$, and satisfy
\[
\left|\left|\frac{\p^k w}{\p\z^k}\right|\right|_{\Har} := \left(\sup_{0< r< 1} \frac{1}{2\pi} \int_{0}^{2\pi} \left|\frac{\p^k w}{\p\z^k}(re^{i\theta})\right|^p \,d\theta\right)^{1/p} < \infty,
\]
for $0 \leq k \leq n-1$.
\end{deff}

Note that the case $f \equiv 0$ and $n = 0$ is the classic holomorphic Hardy spaces $\Har$.  The case $f \equiv 0$ and $n >0$ is the polyanalytic Hardy spaces $H^p_n(D)$. For background on the spaces $H^p_n(D)$, see \cite{Balk}, \cite{polyhardy}, and \cite{polyhardyrep}. For some background on a specific $f \not\equiv 0$ case (Hardy spaces of meta-analytic functions), see \cite{Balk} and \cite{metahardy}. For $f \not\equiv 0$ with respect to a Schwarz boundary value problem comparable to the one we will consider (but without the Hardy integrability condition), see \cite{Beg}. As in the first-order case, we constrain ourselves to the cases of $0 < p \leq 1$. 

To again take advantage of the operators $T(\cdot)$ and $\widetilde{T}(\cdot)$, we will restrict ourselves to the same spaces for $f$ as we did in the first order case. We make that explicit with the following definition.

\begin{deff}
For $0 < p \leq 1$ and $n \geq 1$, we denote by $\Hpqn$ the space of functions $w \in H^{n,p}_f(D)$ where $f \in L^q(D)$, $q>1$.
\end{deff}

We now state the extension of Theorem \ref{main} for the classes $\Hpqn$.

\begin{theorem}\label{mainhigher}
For $\frac{1}{2} < p \leq 1$, $n \geq 1$, and $q>1$, every $w \in \Hpqn$ has a representation of the second kind
\begin{equation}\label{generalsecondkindhigher}
w(z) = \Phi_0(z) + \Psi(z),
\end{equation}
where,
\[
\Psi(z) := \sum_{k = 1}^{n-1}\left(T^k(\Phi_k)(z)\right) + T^n(f)(z),
\]
$\Phi_k \in \Har$, $0 \leq k \leq n-1$,

\[
T(f)(z) := -\frac{1}{\pi} \iint_D \frac{f(\zeta)}{\zeta - z} \,d\xi\,d\eta,
\]
and $T^k(f)$ indicates the $k$-times iteration of the operator $T(\cdot)$ being applied to $f$.
Also, the boundary value in the sense of distributions of $w$, denoted by $w_b$, exists and can be represented as 
\[
w_b = \sum_{\ell} c_\ell a_\ell + \Psi_b.
\]
where $\{c_\ell\} \in \ell^p$, $\{a_\ell\}$ is a collection of $p$-atoms, and $\Psi_b \in L^\gamma(\p D)$, for some $\gamma >1$. 
\end{theorem}

\begin{proof}
By applying Theorem \ref{main}, we have 
\begin{equation*}
\frac{\p^{n-1}w}{\p \z^{n-1}}(z) = \Phi_{n-1}(z) + T(f)(z),
\end{equation*}
where $\Phi_{n-1} \in \Har$ and $T(f) \in L^\gamma(D)$, $1 < \gamma < \frac{2q}{2-q}$. 

By Lemma 1.8.3 from \cite{KlimBook}, since $\Phi_{n-1} \in \Har$, it follows that $\Phi_{n-1} \in L^m(D)$, for all $0 < m < 2p$. Since we are assuming that $\frac{1}{2}< p \leq 1$, it follows that there exists an $s_{n-1}>1$ such that $\Phi_{n-1} \in L^{s_{n-1}}(D)$. Thus, $\Phi_{n-1} + T(f) \in L^{\nu_{n-1}}(D)$, where $\nu_{n-1} := \min\{s_{n-1},\gamma\}$, and $\nu_{n-1} >1$. Hence, we may iterate this process $n-1$ times and appeal to the linearity of the $T(\cdot)$ operator to arrive at the desired representation
\[
w(z) = \Phi_{0}(z) + \sum_{k = 1}^{n-1}\left(T^k(\Phi_k)(z)\right) + T^n(f)(z),
\]
where $\Phi_k \in \Har$, for $k = 0, 1, \ldots, n-1$, $\Psi := \sum_{k = 1}^{n-1}\left(T^k(\Phi_k)(z)\right) + T^n(f)(z)$ is a finite sum of elements of $L^\sigma(D)$, where $\sigma:=\min\{\nu_{k}\}_{k=1}^{n-1}$. Consequently $\Psi$ is an element of $L^\sigma(D)$, and $\sigma > 1$. 

Also, by the exact same argument as that given in the proof of Theorem \ref{main}, we have that the boundary value in the sense of distributions of $w$ is given by
\[
w_b = \Phi_{0,b} + \Psi_{b},
\]
where $\Phi_{0,b}$ is the boundary value in the sense of distributions of $\Phi_0 \in \Har$ and $\Psi_b$ is the boundary value in the sense of distributions of $\Psi$, and $\Psi$ is also an element of $L^\omega(\p D)$ for some $\omega >1$. Since $\Phi_{0,b}$ is a boundary value in the sense of distributions of an element of $\Har$, it follows that $\Phi_{0,b}$ can be represented by an atomic decomposition, so 
\[
w_b = \sum_{\ell} c_\ell a_\ell + \Psi_b,
\]
where $\{c_\ell\} \in \ell^p$ and $\{a_\ell\}$ is a sequence of $p$-atoms.

\end{proof}

In the same way as Theorem \ref{main} was rewritten to Theorem \ref{altmain}, we adapt Theorem \ref{mainhigher} to the following.

\begin{theorem}\label{altmainhigher}
Every $w \in \Hpqn$ can be represented as 
\begin{equation}\label{altsecondrep}
w(z) = \Phi_0(z) + \widetilde{\Psi}(z),
\end{equation}
where 
\[
\widetilde{\Psi}(z) := \widetilde{T}(\Phi_1 + \widetilde{T}(\Phi_2 + \widetilde{T}( \cdots \Phi_{n-1} + \widetilde{T}(f)) \cdots)),
\]
and $\Phi_k \in \Har$, $0 \leq k \leq n-1$.
Also, the boundary value in the sense of distributions of $w$, denoted by $w_b$, exists and can be represented as 
\[
w_b = \Phi_{0,b} + \widetilde{\Psi}_b.
\]
where $\Phi_{0,b}$ is the boundary value in the sense of distributions of $\Phi_0$ and $\widetilde{\Psi}_b$ is the boundary value in the sense of distributions of $\widetilde{\Psi}$. 
\end{theorem}

With the representation provided by Theorem \ref{mainhigher}, the argument for the continuity of the Hilbert transform on the space of boundary values in the sense of distribution follows directly from the argument in the first-order case. 

\begin{deff}
Denote by $(\Hpqn)_b$ the set of distributions $w_b \in \mathcal{D}'(\p D)$ that are boundary values in the sense of distributions of $w \in \Hpqn$. 
\end{deff}

\begin{deff}
We define the quasi-norm $||\cdot||$ by
\[
||w_b|| := ||\sum_{j}c_j a_j ||_{at} + ||\Psi_b||_{L^\gamma(\p D)},
\]
for each $w_b \in (\Hpqn)_b$, $\frac{1}{2} < p \leq 1$, $n \geq 1$, and 
\[
w_b = \sum_{j} c_j a_j + \Psi_b
\]
is the representation from Theorem \ref{mainhigher}, $\gamma$ is the same number as in the statement of Theorem \ref{onefourseven}, and
\[
||\sum_{j} c_j a_j||_{at} := \left( \inf \sum_{j} |c_j|^p \right)^{1/p},
\]
where the infimum is taken over the set of all sequences of coefficients $\{c_j\}$ of equivalent atomic decompositions.
\end{deff}

\begin{theorem}
For $\Hpqn$, where $n \geq 1$, $\frac{1}{2} < p \leq 1$, and $f \in L^q(D)$, $q>1$, the Hilbert transform is a continuous operator on the space $(\Hpqn)_b$.
\end{theorem}

With the representation from Theorem \ref{altmainhigher}, we construct solutions of an analogous Schwarz boundary value problem which are members of the classes $\Hpqn$.

\begin{deff}
We say a function $w: D \to \mathbb{C}$ solves the Schwarz boundary value problem $\mathcal{S}^n_{f,g_0, g_1, \ldots, g_{n-1}}$ whenever $w$ satisfies the following:
\begin{enumerate}
\item $\frac{\p^n w}{\p\z^n} = f$, for some function $f: D \to \mathbb{C}$ and  
\item $\re\left\{ \left(\frac{\p^k w}{\p\z^k}\right)_b \right\} = g_k$ as distributions, where the $g_k \in \mathcal{D}'(\p D)$, $0 \leq k \leq n-1$.
\end{enumerate}
\end{deff}

\begin{theorem}\label{schwarzhigher}
For $n \geq 1$, $f \in L^q(D)$, $q>1$, and $g_k \in \re (H^{p_k}(D))_b$, $\frac{1}{2} < p_k \leq 1$, $0\leq k \leq n -1$,  there exist $w \in \Hpqn$ which solve the boundary value problem $\mathcal{S}^n_{f,g_0, g_1, \ldots, g_{n-1}}$, where $p := \min_k\{p_k\}$. 
\end{theorem}

\begin{proof}
The construction is by iteration of the arguments of the proof of Theorem \ref{schwarzfirst}.
\end{proof}

\begin{corr}\label{atomicdecompforhigher}
For any $w \in H^{p,q}_{n,f}(D)$ that solves $\mathcal{S}^n_{f,g_0, g_1, \ldots, g_{n-1}}$ for some $n \geq 1$, $f \in L^q$, $q>1$, and $g_k \in \re (H^{p_k}(D))_b$, $\frac{1}{2} < p_k  \leq 1$, $0 \leq k \leq n-1$, where $p = \min_k\{p_k\}$, with representation $w = \langle g_0, P_r(\theta - \cdot) + iQ_r(\theta - \cdot)\rangle + \widetilde{\Psi}$ from Theorem \ref{altmainhigher} such that $(\langle g_0, iQ_r(\theta - \cdot)\rangle + \widetilde{\Psi})_b = 0$, $w_b$ has an atomic decomposition. 
\end{corr}

The above corollary gives us conditions which characterize a subset of $\Hpqn$ having boundary values in the sense of distributions that have an atomic decomposition.

\printbibliography

@article{GHJH2,
    author = "Hoepfner, G. and  Hounie, J.",
    title = "Atomic Decompositions of Holomorphic Hardy Spaces in $\mathbb{S}^1$ and Applications",
    journal = "Lecture Notes of Seminario Interdisciplinare di Matematica 7",
    pages = "189-206",
    year = "2008",
}

@book {Vek,
    AUTHOR = {Vekua, I. },
     TITLE = {Generalized analytic functions},
 PUBLISHER = {Pergamon Press, London-Paris-Frankfurt; Addison-Wesley
              Publishing Co., Inc., Reading, Mass.},
      YEAR = {1962},
     PAGES = {xxix+668},
   MRCLASS = {30.81},
  MRNUMBER = {0150320},
MRREVIEWER = {A. E. Heins},
}

@book {KlimBook,
    AUTHOR = {Klimentov, S. },
     TITLE = {{\cyr Granichnye svo\u{i}} {\cyr stva obobshchennykh analiticheskikh funktsi\u{i}}},
    SERIES = {{\cyr Itogi Nauki. Yug Rossii. Matematicheskaya Monografiya} [Progress in
              Science. South Russia. Mathematical Monograph]},
    VOLUME = {7},
 PUBLISHER = {Yuzhny\u{\i} Matematicheski\u{\i} Institut, \\Vladikavkazski\u{\i} Nauchny\u{\i}
              Tsentr, Rossi\u{\i}skaya Akademiya Nauk i RSO-A, Vladikavkaz},
      YEAR = {2014},
     PAGES = {200},
      ISBN = {978-5-904695-26-2},
   MRCLASS = {30-02 (30E25 30G20 30H10 30H15 30H35)},
  MRNUMBER = {3408907},
MRREVIEWER = {Kuzman Adzievski},
}

@article {Klim2016,
    AUTHOR = {Klimentov, S. },
     TITLE = {The {R}iemann-{H}ilbert problem in {H}ardy classes for general
              first-order elliptic systems},
   JOURNAL = {Izv. Vyssh. Uchebn. Zaved. Mat.},
  FJOURNAL = {Izvestiya Vysshikh Uchebnykh Zavedeni\u{\i}. Matematika. Kazanski\u{\i}
              Gosudarstvenny\u{\i} Universitet},
      YEAR = {2016},
    NUMBER = {6},
     PAGES = {36--47},
      ISSN = {0021-3446},
   MRCLASS = {30G20 (30E25 35J56)},
  MRNUMBER = {3674903},
MRREVIEWER = {Armen Grigoryan},
       DOI = {10.3103/s1066369x16060049},
       URL = {https://doi.org/10.3103/s1066369x16060049},
}

@book {BigStein,
    AUTHOR = {Stein, E.},
     TITLE = {Harmonic analysis: real-variable methods, orthogonality, and
              oscillatory integrals},
    SERIES = {Princeton Mathematical Series},
    VOLUME = {43},
      NOTE = {With the assistance of Timothy S. Murphy,
              Monographs in Harmonic Analysis, III},
 PUBLISHER = {Princeton University Press, Princeton, NJ},
      YEAR = {1993},
     PAGES = {xiv+695},
      ISBN = {0-691-03216-5},
   MRCLASS = {42-02 (35Sxx 43-02 47G30)},
  MRNUMBER = {1232192},
MRREVIEWER = {Michael Cowling},
}

@book {Col,
    AUTHOR = {Colzani, L.},
     TITLE = {H{ARDY} {AND} {LIPSCHITZ} {SPACES} {ON} {UNIT} {SPHERES}},
      NOTE = {Thesis (Ph.D.)--Washington University in St. Louis},
 PUBLISHER = {ProQuest LLC, Ann Arbor, MI},
      YEAR = {1982},
     PAGES = {120},
   MRCLASS = {Thesis},
  MRNUMBER = {2632436},
       URL =
              {http://gateway.proquest.com/openurl?url_ver=Z39.88-2004&rft_val_fmt=info:ofi/fmt:kev:mtx:dissertation&res_dat=xri:pqdiss&rft_dat=xri:pqdiss:8302335},
}

@book {Rep,
    AUTHOR = {Mashreghi, J.},
     TITLE = {Representation theorems in {H}ardy spaces},
    SERIES = {London Mathematical Society Student Texts},
    VOLUME = {74},
 PUBLISHER = {Cambridge University Press, Cambridge},
      YEAR = {2009},
     PAGES = {xii+372},
      ISBN = {978-0-521-73201-7},
   MRCLASS = {30-01 (30H10 42-01 42B30 46E15)},
  MRNUMBER = {2500010},
MRREVIEWER = {Marcin M. Bownik},
       DOI = {10.1017/CBO9780511814525},
       URL = {https://doi.org/10.1017/CBO9780511814525},
}

@article {polyhardy,
    AUTHOR = {Wang, Y.},
     TITLE = {On {H}ilbert-type boundary-value problem of poly-{H}ardy class
              on the unit disc},
   JOURNAL = {Complex Var. Elliptic Equ.},
  FJOURNAL = {Complex Variables and Elliptic Equations. An International
              Journal},
    VOLUME = {58},
      YEAR = {2013},
    NUMBER = {4},
     PAGES = {497--509},
      ISSN = {1747-6933},
   MRCLASS = {30E25 (30H10 31A25 45E05)},
  MRNUMBER = {3038743},
MRREVIEWER = {Kuzman Adzievski},
       DOI = {10.1080/17476933.2011.636809},
       URL = {https://doi.org/10.1080/17476933.2011.636809},
}

@book {Balk,
    AUTHOR = {Balk, M.},
     TITLE = {Polyanalytic functions},
    SERIES = {Mathematical Research},
    VOLUME = {63},
 PUBLISHER = {Akademie-Verlag, Berlin},
      YEAR = {1991},
     PAGES = {197},
      ISBN = {3-05-501292-5},
   MRCLASS = {30G20 (31A30)},
  MRNUMBER = {1184141},
MRREVIEWER = {Heinrich Begehr},
}

@article {polyhardyrep,
    AUTHOR = {Du, Z. and Guo, G. and Wang, N.},
     TITLE = {Decompositions of functions and {D}irichlet problems in the
              unit disc},
   JOURNAL = {J. Math. Anal. Appl.},
  FJOURNAL = {Journal of Mathematical Analysis and Applications},
    VOLUME = {362},
      YEAR = {2010},
    NUMBER = {1},
     PAGES = {1--16},
      ISSN = {0022-247X},
   MRCLASS = {31A30 (31A25 31A35)},
  MRNUMBER = {2557663},
MRREVIEWER = {Valery V. Karachik},
       DOI = {10.1016/j.jmaa.2009.07.048},
       URL = {https://doi.org/10.1016/j.jmaa.2009.07.048},
}

@article {metahardy,
    AUTHOR = {Ku, M. and He, F. and Wang, Y.},
     TITLE = {Riemann-{H}ilbert problems for {H}ardy space of meta-analytic
              functions on the unit disc},
   JOURNAL = {Complex Anal. Oper. Theory},
  FJOURNAL = {Complex Analysis and Operator Theory},
    VOLUME = {12},
      YEAR = {2018},
    NUMBER = {2},
     PAGES = {457--474},
      ISSN = {1661-8254},
   MRCLASS = {30E25 (30H10)},
  MRNUMBER = {3756167},
MRREVIEWER = {Kuzman Adzievski},
       DOI = {10.1007/s11785-017-0705-1},
       URL = {https://doi.org/10.1007/s11785-017-0705-1},
}

@book {Duren,
    AUTHOR = {Duren, P.},
     TITLE = {Theory of {$H^{p}$} spaces},
    SERIES = {Pure and Applied Mathematics, Vol. 38},
 PUBLISHER = {Academic Press, New York-London},
      YEAR = {1970},
     PAGES = {xii+258},
   MRCLASS = {46.30 (30.00)},
  MRNUMBER = {0268655},
MRREVIEWER = {D. Sarason},
}

@book {BegBook,
    AUTHOR = {Begehr, H.},
     TITLE = {Complex analytic methods for partial differential equations},
      NOTE = {An introductory text},
 PUBLISHER = {World Scientific Publishing Co., Inc., River Edge, NJ},
      YEAR = {1994},
     PAGES = {x+273},
      ISBN = {981-02-1550-9},
   MRCLASS = {35C15 (30E20 35A20)},
  MRNUMBER = {1314196},
MRREVIEWER = {Guo Chun Wen},
       DOI = {10.1142/2162},
       URL = {https://doi.org/10.1142/2162},
}

@article {Beg,
    AUTHOR = {Begehr, H. and Schmersau, D.},
     TITLE = {The {S}chwarz problem for polyanalytic functions},
   JOURNAL = {Z. Anal. Anwendungen},
  FJOURNAL = {Zeitschrift f\"{u}r Analysis und ihre Anwendungen. Journal for
              Analysis and its Applications},
    VOLUME = {24},
      YEAR = {2005},
    NUMBER = {2},
     PAGES = {341--351},
      ISSN = {0232-2064},
   MRCLASS = {30E25 (30G20 31A30 35J40)},
  MRNUMBER = {2174027},
MRREVIEWER = {Wolfgang Tutschke},
       DOI = {10.4171/ZAA/1244},
       URL = {https://doi.org/10.4171/ZAA/1244},
}

@book {ellipquasi,
    AUTHOR = {Astala, K. and Iwaniec, T. and Martin, G.},
     TITLE = {Elliptic partial differential equations and quasiconformal
              mappings in the plane},
    SERIES = {Princeton Mathematical Series},
    VOLUME = {48},
 PUBLISHER = {Princeton University Press, Princeton, NJ},
      YEAR = {2009},
     PAGES = {xviii+677},
      ISBN = {978-0-691-13777-3},
   MRCLASS = {30C62 (30G20 35J46 35J60 35J92)},
  MRNUMBER = {2472875},
MRREVIEWER = {Olli Martio},
}

@book {Koosis,
    AUTHOR = {Koosis, P.},
     TITLE = {Introduction to {$H_p$} spaces},
    SERIES = {Cambridge Tracts in Mathematics},
    VOLUME = {115},
   EDITION = {Second},
      NOTE = {With two appendices by V. P. Havin [Viktor Petrovich Khavin]},
 PUBLISHER = {Cambridge University Press, Cambridge},
      YEAR = {1998},
     PAGES = {xiv+289},
      ISBN = {0-521-45521-9},
   MRCLASS = {30D55 (42B30 46E15 46J15)},
  MRNUMBER = {1669574},
MRREVIEWER = {D. Sarason},
}

@article {Coif1,
    AUTHOR = {Coifman, R.},
     TITLE = {A real variable characterization of {$H^{p}$}},
   JOURNAL = {Studia Math.},
  FJOURNAL = {Polska Akademia Nauk. Instytut Matematyczny. Studia
              Mathematica},
    VOLUME = {51},
      YEAR = {1974},
     PAGES = {269--274},
      ISSN = {0039-3223},
   MRCLASS = {46E15},
  MRNUMBER = {358318},
MRREVIEWER = {J. A. van Casteren},
       DOI = {10.4064/sm-51-3-269-274},
       URL = {https://doi.org/10.4064/sm-51-3-269-274},
}

@article {Coif2,
    AUTHOR = {Coifman, R. and Weiss, G.},
     TITLE = {Extensions of {H}ardy spaces and their use in analysis},
   JOURNAL = {Bull. Amer. Math. Soc.},
  FJOURNAL = {Bulletin of the American Mathematical Society},
    VOLUME = {83},
      YEAR = {1977},
    NUMBER = {4},
     PAGES = {569--645},
      ISSN = {0002-9904},
   MRCLASS = {42A40 (30A78 42A18 43A85)},
  MRNUMBER = {447954},
MRREVIEWER = {Alberto Torchinsky},
       DOI = {10.1090/S0002-9904-1977-14325-5},
       URL = {https://doi.org/10.1090/S0002-9904-1977-14325-5},
}
\end{document}